\documentclass[12pt]{amsart}
\usepackage[dvips]{graphicx,color}
\usepackage{amssymb,amscd,latexsym,epsfig,hyperref}
\usepackage[mathscr]{euscript}
\usepackage[curve]{xypic}

\DeclareFontFamily{OT1}{rsfs}{}
\DeclareFontShape{OT1}{rsfs}{n}{it}{<-> rsfs10}{}
\DeclareMathAlphabet{\curly}{OT1}{rsfs}{n}{it}

\renewcommand\;{\hspace{.6pt}}

\newcommand\PP{\mathbb P}
\newcommand\LL{\mathbb L}
\newcommand\C{\mathbb C}
\newcommand\Q{\mathbb Q}

\newcommand\Z{\mathbb Z}

\newcommand\cO{\mathcal O}

\newcommand\cE{\mathcal E}

\newcommand\cI{\mathcal I}

\newcommand\cM{\mathcal M}

\newcommand\cP{\mathcal P}

\newcommand\cS{\mathcal S}

\newcommand\udot{^{\bullet}}

\makeatletter
\newcommand{\so}{\ \ext@arrow 0359\Rightarrowfill@{}{\hspace{3mm}}\ }
\makeatother
\newcommand{\rt}[1]{\xrightarrow{\ #1\ }}
\newcommand\To{\longrightarrow}

\newcommand\into{\hookrightarrow}
\newcommand\INTO{\ \ar@{^(->}[r]<-.2ex>}
\newcommand{\Into}{\ensuremath{\lhook\joinrel\relbar\joinrel\rightarrow}}

\renewcommand\_{^{}_}
\newcommand\take{\backslash}

\newfont{\bigtimesfont}{cmsy10 scaled \magstep5}
\newcommand{\bigtimes}{\mathop{\lower0.9ex\hbox{\bigtimesfont\symbol2}}}
\renewcommand\={\ =\ }

\renewcommand\t{\mathfrak t}
\newcommand\vw{\mathsf{vw}}
\newcommand\VW{\mathsf{VW}}

\newcommand\vir{\operatorname{vir}}
\newcommand\red{\operatorname{red}}

\newcommand\tr{\operatorname{tr}}
\newcommand\coker{\operatorname{coker}}

\newcommand\id{\operatorname{id}}

\newcommand\Hom{\operatorname{Hom}}
\renewcommand\hom{\curly H\!om}

\newcommand\Ext{\operatorname{Ext}}

\newcommand\Hilb{\operatorname{Hilb}}

\newcommand\beq[1]{\begin{equation}\label{#1}}
\newcommand\eeq{\end{equation}}
\newcommand\beqa{\begin{eqnarray*}}
\newcommand\eeqa{\end{eqnarray*}}

\newcommand\arXiv[1]{\href{http://arxiv.org/abs/#1}{arXiv:#1}}
\newcommand\mathAG[1]{\href{http://arxiv.org/abs/math/#1}{math.AG/#1}}
\newcommand\hepth[1]{\href{http://arxiv.org/abs/hep-th/#1}{hep-th/#1}}

\DeclareRobustCommand{\SkipTocEntry}[3]{}
\makeatletter
\newcommand\@dotsep{4.5}
\def\@tocline#1#2#3#4#5#6#7{\relax
  \ifnum #1>\c@tocdepth 
  \else
    \par \addpenalty\@secpenalty\addvspace{#2}%
    \begingroup \hyphenpenalty\@M
    \@ifempty{#4}{%
      \@tempdima\csname r@tocindent\number#1\endcsname\relax
    }{%
      \@tempdima#4\relax
    }%
    \parindent\z@ \leftskip#3\relax \advance\leftskip\@tempdima\relax
    \rightskip\@pnumwidth plus1em \parfillskip-\@pnumwidth
    #5\leavevmode #6\relax
    \leaders\hbox{$\m@th
      \mkern \@dotsep mu\hbox{.}\mkern \@dotsep mu$}\hfill
    \hbox to\@pnumwidth{\@tocpagenum{#7}}\par
    \nobreak
    \endgroup
  \fi}
\makeatother

\makeatletter \@addtoreset{equation}{section} \makeatother

\newtheorem{thm}[equation]{Theorem}
\newtheorem{conj}[equation]{Conjecture}
\newtheorem{lem}[equation]{Lemma}
\newtheorem{cor}[equation]{Corollary}
\newtheorem{prop}[equation]{Proposition}
\newenvironment{rmk}{\noindent\textbf{Remark}.}{\\}

\title{Sheaf counting on local K3 surfaces}
\author{Davesh Maulik and Richard P. Thomas}

\begin{document}
\maketitle
\begin{abstract} \noindent
There are two natural ways to count stable pairs or Joyce-Song pairs on $X=\mathrm{K3}\times\C$; one via weighted Euler characteristic and the other by virtual localisation of the reduced virtual class. Since $X$ is noncompact these need not be the same. We show their generating series are related by an exponential.

As applications we prove two conjectures of Toda, and a conjecture of Tanaka-Thomas defining Vafa-Witten invariants in the semistable case.
\end{abstract}


\tableofcontents

\section{Introduction}
Let $S$ be a smooth complex projective K3 surface, and $X=S\times\C$. Let $\mathsf Z^{\red}_P$ be the generating series of reduced residue stable pair invariants of $X$, and let $\mathsf z^\chi_P$ be the generating series of weighted Euler characteristics of stable pairs moduli spaces. Precise definitions are given in the next section.

\begin{thm} \label{it}
$$
\mathsf Z^{\red}_P(X,q,v)\=-\log\big(1+\mathsf z^\chi_P(X,q,v)\big).
$$
\end{thm}

The idea of the proof is to start with $Y=S\times E$, where $E$ is an elliptic curve. To get nonzero invariants we divide the moduli space by the translation action of $E$ and use Oberdieck's symmetric reduced obstruction theory \cite{Ob}. Since the moduli space is compact, the invariants defined by virtual cycle or weighted Euler characteristic coincide:
$$
\mathsf Z^{\red}_P(Y/E,q,v)\=\mathsf z^\chi_P(Y/E,q,v).
$$
On the left hand side we then use Li-Wu's degeneration formula for virtual cycles \cite{LW} as $E$ degenerates to a rational nodal curve. This has a $\C^*$ action; applying virtual localisation ultimately gives the left hand side of Theorem \ref{it}. On the right hand side we work with weighted Euler characteristics, using a simple gluing argument to compare the moduli spaces of stable pairs supported set-theoretically on one K3 fibre of either $X$ or $Y$. An elementary calculation of Euler characteristics of configuration spaces of points on a punctured elliptic curve then gives the right hand side of Theorem \ref{it}. \medskip

Since $\mathsf Z^{\red}_P(X)$ is now proved to be given by the KKV formula \cite{PT3}, this implies a weighted Euler characteristic version of the KKV conjecture for $X$. We also prove an unweighted, bare Euler characteristic version, thus proving a conjecture of Toda \cite{To1}. From this we also deduce a multiple cover formula conjectured in \cite{To2} for invariants counting 1-dimensional semistable sheaves on $X$.
\medskip

Replacing stable pairs by Joyce-Song pairs of arbitrary semistable sheaves on $X$, a similar argument relates their virtual and weighted Euler characteristic invariants. Combined with Joyce-Song's universal identity for the weighted Euler characteristic invariants, this proves an identity (roughly the logarithm of Joyce-Song's identity) conjectured in \cite{TT2}. This is important for the  conjectural definition \cite{TT2} of Vafa-Witten invariants of surfaces $S$ in the presence of strict semistables.

\begin{thm} For $S$ a K3 surface \cite[Conjecture 7.2]{TT2} holds, so the Vafa-Witten invariants $\VW_\alpha(S)$ are well defined. Moreover, they equal the invariants $\vw_\alpha(S)$ defined by weighted Euler characteristic in \mbox{\cite[Section 4]{TT2}.}
\end{thm}

The analogous result was proved for surfaces with $\deg K_S<0$ in \cite{TT2}. The Conjecture was also proved there for \emph{all} surfaces when there are no strictly semistable sheaves in class $\alpha$, but when $K_S>0$ the invariants $\VW_\alpha(S)$ and $\vw_\alpha(S)$ no longer coincide.

\medskip\noindent\textbf{Acknowledgements.} Many thanks to Yukinobu Toda, Rahul Pandharipande, Georg Oberdieck, J{\o}rgen Rennemo, Ed Segal and a very insightful referee for useful comments and conversations.  DM is supported by NSF grants DMS-1645082 and DMS-1564458. RPT acknowledges partial support from EPSRC grant EP/R013349/1.

\section{Notation}
Throughout we will use the following notation.
\begin{itemize}
\item $S$ is an algebraic K3 surface over $\C$.
\item $H^{1,1}(S,\Z):=H^2(S,\Z)\cap H^{1,1}(S,\C)$ inside $H^2(S,\C)$. We freely identify $H^2(S)\cong H_2(S)$ over $\Z,\,\Q$ or $\C$, by Poincar\'e duality.
\item $\beta\in H^{1,1}(S,\Z)$ is a curve class.
\item $X=K_S=S\times\C$ is a Calabi-Yau 3-fold with $\C^*$ action of weight 1 on the $\C$ factor.
\item $Y=S\times E$ for some smooth elliptic curve $E$ with origin $e\in E$.
\item $T$ is a 3-fold, K3-fibred over the marked disk $(\Delta,0)$ with central fibre $S$. Considered as a map from $\Delta$ to the moduli stack of K3 surfaces, it intersects the Noether-Lefschetz divisor of $\beta$ only in $0$, and \emph{transversely}. Moreover it is also transverse to Noether-Lefschetz divisors corresponding to all all classes $\beta'\in H^{1,1}(S,\Z)$ of smaller degree. We think of $T$ as an algebraic approximation to a twistor space for $S$. For full details see \cite[Section 6.2]{PT3}.
\item We use $\pi$ for the three K3 fibrations $X\to\C,\ Y\to E,\ T\to\Delta$.
\item We use $\iota\colon S\into X,Y,T$ for the three inclusions of the central fibre $S$ (over $0\in\C,\ e\in E,\ 0\in\Delta$ respectively) in the above fibrations.
\end{itemize}

We study stable pairs $(F,s)$ on $X=S\times\C$; that is,
\begin{enumerate}
\item $F$ is a coherent sheaf on $X$ of pure dimension one, and
\item $s\in H^0(F)$ has finite cokernel.
\end{enumerate}
For $n\in\Z$, there is a quasi-projective moduli space
$$
P_n(X,\iota_*\beta)\ =\ P_n(S\times\C,\iota_*\beta)
$$
of stable pairs with holomorphic Euler characteristic $\chi(F)=n$ and curve class $[F]=\iota_*\beta$. Since it admits a symmetric obstruction theory there are two ways of extracting invariants from this moduli space, and since it is noncompact they need not be equal.
\medskip

We define the generating series of (integer) Kai-weighted Euler characteristic invariants by
\beq{zp}
\mathsf z_P^\chi(X,q,v)\ :=\ \sum_{\beta,n}e\big(P_n(X,\iota_*\beta),\chi^B\big)\;q^nv^\beta.
\eeq
Here $\chi^B$ is Behrend's integer-valued constructible function of the scheme $P_n(X,\iota_*\beta)$ \cite{Be}. We will compare this to the generating series
\beq{rr2}
\mathsf Z_P^{\red}(X,q,v)\ :=\ \sum_{\alpha}P^{\mathrm{red}}_{n,\beta}(X)q^nv^\beta
\eeq
of (rational) \emph{reduced residue} stable pair invariants of $S$. These are defined by applying the virtual localisation formula \cite{GP} to the \emph{reduced} perfect obstruction theory of $X=S\times\C$ acted on by $\C^*$,
\beq{rr}
P^{\mathrm{red}}_{n,\beta}(X)\ :=\ 
\mathrm{Res}\_{\;t=0}\int_{[P_n(X,\,\iota_*\beta)^{\C^*}]^{\mathrm{red}}}\,\frac 1{e(N^{\mathrm{vir}})}\ \in\ \Q.
\eeq
Here $t\in H^2(B\C^*)$ is the equivariant parameter. \medskip

Note that by condition (1) above, only charges $(\beta,n)$ with $\beta\ne0$ contribute to the sums \eqref{zp} and \eqref{rr2}.
It will be convenient to abbreviate the charge (pushed down to $S$) to
$$
\alpha\,:=\,(\beta,n)\ \in\ H^{1,1}(S,\Z)\oplus\Z
$$
and write
\beq{kaiinvt}
P_\alpha\,:=\,e\big(P_n(X,\iota_*\beta),\chi^B\big)
\quad\mathrm{and}\quad Q^\alpha\,:=\,q^nv^\beta
\eeq
so that \eqref{zp}, for instance, becomes
\beq{zp4}
\mathsf z_P^\chi(X,Q)\=\sum_\alpha P_\alpha\,Q^\alpha.
\eeq
To prove Theorem \ref{it} relating the two generating series $\mathsf z^\chi_P$ and $\mathsf Z_P^{\red}$ we first relate both to the invariants of $Y=S\times E$.

\section{Behrend}\label{Kai}
Let $P_n(Y,\iota_*\beta)$ denote the moduli space of stable pairs on $Y$ in the curve class $\iota_*\beta$. Since it carries an action of $E$ with finite stabilisers, we can consider the weighted Euler characteristic of the quotient $P_n(Y,\iota_*\beta)/E$, where the weighting is by $\chi^B$ \emph{divided by the order of the stabiliser group} at any point. We form the generating series
\beq{Ygen}
\mathsf z_P^\chi\big(Y/E,Q\big)\ :=\ \sum_{\alpha}e\big(P_n(Y,\iota_*\beta)/E,\chi^B\big)Q^\alpha.
\eeq
The relationship between the weighted Euler characteristic invariants of $X$ \eqref{zp} and $Y$ \eqref{Ygen} is the following.

\begin{prop} \label{jpo}
$$
\mathsf z_P^\chi\big(Y/E,q,v\big)\=-\log\big(1+\mathsf z_P^\chi(X,q,v)\big).
$$
\end{prop}

\begin{proof}
Firstly we use the $\C^*$ action on $X=S\times\C$, which induces one on the moduli spaces $P_n(X,\iota_*\beta)$ preserving the Behrend function $\chi^B$. Its fixed points lie in the locus 
$$
P^0_\alpha\ \subset\ P_n(X,\iota_*\beta)
$$
of stable pairs with set-theoretic support on $S\times\{0\}\subset S\times\C$. The $\C^*$ action is free on the complement of this locus, so its weighted Euler characteristic is zero. Therefore \eqref{kaiinvt} localises to
\beq{loc}
P_\alpha\ =\ e\Big(P^0_\alpha,\,\chi^B|_{P^0_\alpha}\Big).
\eeq
(The invariant $P_\alpha$ can be further localised to the $\C^*$-fixed points of $P^0_\alpha$, but we do not use that here.) \medskip

To relate $X$ and $Y$ we fix, once and for all, a trivialisation of the tangent bundle of $E$. The exponential map then gives a canonical analytic isomorphism between a neighbourhood of any point $p\in E$ and a neighbourhood of $0\in\C$. Using this to transplant stable pairs from $S\times\{0\}\subset X$ to $S\times\{p\}\subset Y$ we see
$$
P^0_{\alpha}\times E
$$
as the moduli space of stable pairs on $Y=S\times E$ (with holomorphic Euler characteristic $n$ and curve class $\iota_*\beta$) \emph{supported set theoretically on a single fibre $S\times\{p\}$}. \medskip

This suggests stratifying the moduli space $P_n(Y,\iota_*\beta)$ by the (minimal) number $k$ of fibres $S$ on which the pairs are set-theoretically supported. Each carries a charge; we call the \emph{distinct} charges $\alpha_1,\ldots,\alpha_\ell$. Let $k_i$ denote the number of fibres $S$ carrying charge $\alpha_i=(\beta_i,n_i)$, so
$$
\sum_{i=1}^\ell k_i\=k\quad\mathrm{and}\quad\sum_{i=1}^\ell k_i\alpha_i\=\alpha.
$$
We claim the stratum of the moduli space $P_n(Y,\iota_*\beta)$ with this data is the quotient of
\beq{mod}
\underbrace{P^0_{\alpha_1}\times\cdots\times P^0_{\alpha_1}}_{k_1}\times\cdots\times \underbrace{P^0_{\alpha_\ell}\times\cdots\times P^0_{\alpha_\ell}}_{k_\ell}\times(E^k\take\Delta_k)
\eeq
by the free action of the product $\Sigma_{k_1}\times\cdots\times\Sigma_{k_\ell}$ of symmetric groups. Here $\Delta_k$ is the \emph{big} diagonal.

To see the claim, observe the product \eqref{mod} is the moduli space of pairs with a fixed choice of ordering of the $k_i$ points of $E$ above which the charges $\alpha_i$ are supported. Since these ordered points are distinct they are equivalent to a point of $E^k\take\Delta_k$. The symmetric groups act by permuting these points, changing their ordering.

Since the Behrend function is determined analytically, it cannot tell the difference between $\C$ and $E$. Therefore its pull back from $P_n(Y,\iota_*\beta)$ to \eqref{mod} is just the product of the Behrend functions
$$
\chi^B_{P_{n_i}(X,\,\iota_*\beta_i)}\big|_{P^0_{\alpha_i}}
$$
on each $P^0_{\alpha_i}$ factor. Dividing by the symmetric groups preserves the Behrend function $\chi^B$, while dividing by $E$ just changes its sign. Thus the weighted Euler characteristic of the quotient of \eqref{mod} by the action of both the symmetric groups and $E$ is, by \eqref{loc},
$$
-\frac{P_{\alpha_1}^{k_1}}{k_1!}\frac{P_{\alpha_2}^{k_2}}{k_2!}\cdots\frac{P_{\alpha_\ell}^{k_\ell}}{k_\ell!}
\,e\!\left(\!\frac{E^k\take\Delta_k}E\!\right)\=(-1)^k\frac1k{k\choose k_1,k_2,\ldots,k_\ell}P_{\alpha_1}^{k_1}\cdots P_{\alpha_\ell}^{k_\ell}\,.
$$
Here we have used
$$
e\left(\!\frac{E^k\take\Delta_k}E\!\right)\=(-1)^{k-1}(k-1)!\,,
$$
proved inductively using the fibration $(E^k\take\Delta_k)/E\to(E^{k-1}\take\Delta_{k-1})/E$, whose fibre $E\take\{p_1,\ldots,p_{k-1}\}$ has Euler characteristic $-(k-1)$.

Summing over all strata and all $\alpha$ computes the generating series \eqref{Ygen} as
\begin{align*}
\sum_{k=1}^\infty\ \ \mathop{\sum_{\ell,k_i,\,\alpha_i\mathrm{\,distict}\colon}}_{\sum_{i=1}^\ell k_i=k}\ &\frac{(-1)^k}k{k\choose k_1,k_2,\ldots,k_\ell}(P_{\alpha_1}Q^{\alpha_1})^{k_1}\cdots (P_{\alpha_\ell}Q^{\alpha_\ell})^{k_\ell} \\
=\ \sum_{k=1}^\infty\frac1k&\left(-\sum_{\alpha}P_\alpha Q^\alpha\right)^{\!\!k}
\\ &=\ -\log\left(1+\sum_{\alpha}P_\alpha Q^\alpha\right)\!,
\end{align*}
which by \eqref{zp4} is what we wanted to prove.
\end{proof}

\begin{rmk}
It is possible to give a quicker, more highbrow proof of this result using the technology introduced by Oberdieck-Shen \cite{OS}. Using their $E$-equivariant power structure, the stratification of $P_n(Y,\iota_*\beta)$ used above \eqref{mod} is equivalent to the identity
$$
1+\sum_{n,\beta}\big[P_n(Y,\iota_*\beta)\big] q^n v^{\beta}\=\Big(
1+\sum_{\alpha}[P_{\alpha}^0]Q^\alpha\Big)^{\![E]}
$$
in the $E$-equivariant Grothendieck group of varieties $K_0^E(\mathrm{Var})[\![q,v]\!]$. Applying (a Behrend-weighted version of) their $E$-equivariant integration map $\cI$ to the ring of dual numbers $\Q[\epsilon]/(\epsilon^2)$, we get
\beqa
1+\epsilon\sum_{n,\beta}e\big(P_n(Y,\iota_*\beta)/E,\chi^B\big)q^nv^{\beta}
&=& \Big(1+\sum_\alpha P_\alpha Q^\alpha\Big)^{\cI(E)} \\
&=& \Big(1+\sum_\alpha P_\alpha Q^\alpha\Big)^{\!-\epsilon} \\
&=& 1-\epsilon\log\Big(1+\sum_\alpha P_\alpha Q^\alpha\Big).
\eeqa
Taking coefficients of $\epsilon$ then recovers Proposition \ref{jpo}.
\end{rmk}

\section{Oberdieck} \label{Georg}
Oberdieck \cite[Section 3]{Ob} proves that $P_n(Y,\iota_*\beta)/E$ carries a natural symmetric perfect obstruction theory, so that its Kai-weighted invariants \eqref{Ygen} of the last Section coincide with invariants defined by integrating 1 over his virtual cycle. Furthermore, in \cite[Theorem 1]{Ob} he proves that the latter invariants can be defined differently --- by taking the \emph{reduced} virtual cycle on $P_n(Y,\iota_*\beta)$ and integrating an insertion instead of dividing by $E$.

That is, letting $\beta^\vee\in H^2(S,\Q)$ denote any class with $\int_S\beta\cup\beta^\vee=1$, we have
\beq{BP}
e\big(P_n(Y,\iota_*\beta)/E,\chi^B\big)\=\int_{[P_n(Y,\iota_*\beta)]^{\mathrm{red}}}
\tau_0(\iota_*\beta^\vee).
\eeq
To the right hand side of \eqref{BP} we can now apply the degeneration formula \cite{Li, LW}, since that uses virtual cycles rather than weighted Euler characteristics. Degenerating $E$ to be a 1-nodal rational elliptic curve will express \eqref{BP} in terms of the reduced residue stable pair invariants of $S\times\C=X$ \eqref{rr}. Combined with the calculation of the left hand side of \ref{BP} in Proposition \ref{jpo} we will obtain the following.

\begin{thm}\label{DM} The reduced residue stable pair invariants of $X$ \eqref{rr2} are related to its weighted Euler characteristic stable pair invariants \eqref{zp} by 
$$
\mathsf Z_P^{\red}(X,q,v)\=-\log\big(1+\mathsf z^\chi_P(X,q,v)\big).
$$
\end{thm}

\begin{proof}
Taking $E$ to be a 1-nodal rational elliptic curve,
the degeneration formula \cite{Li, LW} expresses the right hand side of \eqref{BP} as
\beq{form}
\int_{\big[P_n\big((S\times\PP^1)/(S_0\cup S_\infty),\,\iota_*\beta\big)\big]^{\mathrm{red}}}
\tau_0(\iota_*\beta^\vee).
\eeq
That is, we work with stable pairs on $S\times\PP^1$ relative to the divisors $S_0:=S\times\{0\}$ and $S_\infty:=S\times\{\infty\}$. The fibre product (matching the stable pairs over $S_0$ with those over $S_\infty$ so they can be glued together) imposes no condition since in the class $\iota_*\beta$ any relative stable pair is canonically trivial on the relative divisors.

We now calculate \eqref{form} as in \cite[Sections 7.3--7.4]{PT3}. We start by considering the relative geometry $S\times\PP^1/S_\infty$ and degenerating it to
$$
(S\times\PP^1)/S_\infty\ \mathop{\cup}_{S_\infty\sim S_0}\ (S\times\PP^1)/(S_0\cup S_\infty).
$$
The degeneration formula gives
\begin{multline*}
\int_{\big[P_n\big((S\times\PP^1)/S_\infty,\,\iota_*\beta\big)\big]^{\mathrm{red}}}
\tau_0(\iota_*\beta^\vee)\= \\
\sum\int_{\big[P_{n_1}\big((S\times\PP^1)/S_\infty,\,\iota_*\beta_1\big)\times P_{n_2}\big((S\times\PP^1)/(S_0\cup S_\infty),\,\iota_*\beta_2\big)\big]^{\mathrm{red}}}
1\times\tau_0(\iota_*\beta^\vee),
\end{multline*}
where the sum is over all $(\beta_1,n_1),\,(\beta_2,n_2)$ whose sum is $(\beta,n)$. If both of $(\beta_1,n_1),\,(\beta_2,n_2)$ are nonzero then the obstruction theory admits two surjective cosections corresponding to the two components of support of the stable pair. Since the reduced obstruction theory removes only one, it still admits one surjective cosection, so the reduced cycle is zero.

When $(\beta_1,n_1)=(\beta,n)$ and $(\beta_2,n_2)=0$ the integral is zero because of the insertion. So we are left only with $(\beta_2,n_2)=(\beta,n),\ (\beta_1,n_1)=0$, which contributes \eqref{form}. Therefore, by \eqref{BP},
$$
e\big(P_n(Y,\iota_*\beta)/E,\chi^B\big)\=
\int_{\big[P_n\big((S\times\PP^1)/S_\infty,\,\iota_*\beta\big)\big]^{\mathrm{red}}}
\tau_0(\iota_*\beta^\vee).
$$
Then we apply virtual localisation to the usual $\C^*$ action on $\PP^1$ with weight $+1$ on the tangent space at 0. We lift $\iota_*\beta^\vee$ to $H^4_{\C^*}(S\times\PP^1)$ by letting $\iota$ be the inclusion of the $\C^*$-invariant divisor $S\times\{0\}\into S\times\PP^1$.

By the same double cosection argument the contributions vanish unless all the charge $(\beta,n)$ is at one of $S_0$ or (bubbles over) $S_\infty$. And again the insertion $\tau_0(\iota_*\beta^\vee)$ kills the contribution of pairs supported entirely in (bubbles over) $S_\infty$. Therefore we localise everything to $S_0$ where the result becomes
$$
\int_{[P_n(X,\,\iota_*\beta)^{\C^*}]^{\mathrm{red}}}\,\frac1{e(N^{\mathrm{vir}})}\,\tau_0(\iota_*\beta^\vee).
$$
Over $S_0$ the insertion $\tau_0(\iota_*\beta^\vee)$ is just $c_1(\t)\int_S\beta\cup\beta^\vee=t$, the equivariant parameter, so we end up with
$$
\int_{[P_n(X,\,\iota_*\beta)^{\C^*}]^{\mathrm{red}}}\,\frac t{e(N^{\mathrm{vir}})}\=\mathrm{Res}\_{\;t=0}\int_{[P_n(X,\,\iota_*\beta)^{\C^*}]^{\mathrm{red}}}\,\frac 1{e(N^{\mathrm{vir}})}\=
P_{n,\beta}^{\mathrm{red}}(X).
$$
Combining with Proposition \ref{jpo} gives the result.
\end{proof}

\section{Euler}
Theorem \ref{DM} relates $\mathsf Z_P^{\red}(X,q,v)$ and $\mathsf z^\chi_P(X,q,v)$. The former has now been computed by the KKV formula \cite{PT3}. For the latter, we have Toda's conjectural weighted Euler characteristic version of the KKV conjecture \cite{To1}. In fact Toda works with  ``naive" unweighted stable pair invariants defined by bare Euler characteristic. So to deduce his conjecture we need to relate the naive and weighted invariants using a version of dimensional reduction.

By localisation, it is sufficient to show
the Behrend function is always $\pm1$ at $\C^*$-fixed points of the moduli space of stable pairs on $X=K_S=S\times\C$. In this section we show how to do this, though we only sketch the derived stacks technicalities involved. By now Toda has given a much more careful and professional treatment in \cite[Section 6]{To3}, in fact proving more general results. \medskip


We start with $\cM_S$, the moduli stack of pure coherent sheaves $E$ on $S$ with support of dimension $1$ in some fixed curve class $\beta\in H_2(S)$, and fixed holomorphic Euler characteristic $n\gg0$. This has a natural derived Artin stack structure \cite{TV}. There is a closely related derived stack $\cP_S$ of pairs $(E,s)$, where $s\in H^0(E)$. The fibre over $E\in\cM_S$ of the forgetful map
\beq{fibr}
\cP_S\ \To\ \cM_S
\eeq
has underlying scheme $H^0(E)$ and virtual relative tangent bundle $R\Gamma(E)$. There is a corresponding exact triangle of virtual tangent bundles
$$
R\Gamma(E)\To R\Hom(I\udot_S,E)\To R\Hom(E,E)[1]
$$
at $(E,s)$, where $I\udot_S$ is the complex $\cO_S\rt{s}E$ (with $\cO_S$ in degree 0).

Let $\cP^0_S\subset\cP_S$ denote the open substack of pairs $(E,s)$ for which $\Ext^1(E,E)\rt{s}H^1(E)$ onto. Combined with the vanishing of $H^2(E)$ (since $E$ has dimension $1$) this makes the $\Ext^*(\ \cdot\ ,E)$ exact sequence of the exact triangle $I\udot_S\to\cO_S\to E$ collapse to give
\beq{obzz}
\Ext^1(I\udot_S,E)\ \cong\ \Ext^2(E,E)\ \cong\ \Hom(E,E\otimes K_S)^*,
\eeq
the last isomorphism being Serre duality. \medskip

Now let $X:=K_S\rt{p}S$ be the total space of the canonical bundle of $S$, and let $P_X=P_n(X,\iota_*\beta)$ denote the moduli \emph{scheme} of stable pairs $(\cE,s)$ of class $(\iota_*\beta,n)$. Pushing down defines a map $p_*\colon P_X\to\cP_S$ taking $(\cE,s)$ to $(E:=p_*\;\cE,p_*s)$.

\begin{lem} The map $p_*$ factors through $\cP_S^0\subset\cP_S$.
\end{lem}

\begin{proof} (Cf. \cite[Proof of Proposition C.2]{PT2}.) 
We need to show that $\Ext^1(E,E)\rt{p_*s} H^1(E)$ is onto, where $E=p_*\;\cE$.

Let $C\subset X$ be the support of $\cE$ and consider the maps
\begin{multline*}
H^1(\cO_C)\rt{\id_\cE}H^1_X(\hom(\cE,\cE))\Into\Ext^1_X(\cE,\cE)\To
\\ \Ext^1_S(E,E)\rt{p_*s}H^1(E)\ \cong\ H^1(\cE).
\end{multline*}
The third arrow is given by adjunction and the evaluation map $p^*E=p^*p_*\;\cE\to\cE$. The composition is multiplication by $s$, and is therefore onto since dim\,coker\,$s=0$ implies $H^1(\coker s)=0$. Thus the final arrow $p_*s$ must also be onto.
\end{proof}

The projection $p_*\colon P_X\to\cP_S^0$ in fact exhibits $P_X$ as an open substack (with trivial stabilisers) of the $(-1)$-shifted cotangent bundle
\beq{shift}
P_X\ \subset\ T^*[-1]\cP_S^0.
\eeq
As a scheme $T^*[-1]\cP^0_S$ is the total space of the dual obstruction sheaf of $\cP_S^0$, with closed points given by  triples $(E,s,\phi)$,
$$
E\in\mathrm{Coh}(S), \quad s\in H^0(E), \quad \phi\in\Hom(E,E\otimes K_S)\cong\Ext^1(I\udot_S,E)^*
$$
by \eqref{obzz}. This triple is equivalent to a pair $(\cE_\phi,s)$ on $X$ by the spectral construction. That is, $\cE_\phi$ on $K_S=X$ is the eigensheaf\;\footnote{The action of $\phi$ makes the $\cO_S$-module $E$ into a Sym${}^\bullet K_S^{-1}=p_*\cO_X$-module.} of the Higgs field $\phi$, and $s\in H^0(E)=H^0(p_*\;\cE_\phi)=H^0(\cE_\phi)$.

The derived scheme structure on $P_X$ pulled back from \eqref{shift} is quasi-smooth, inducing the usual (\emph{not} reduced!) symmetric perfect obstruction theory on $P_X$. Moreover the description \eqref{shift}
expresses $P_X$ as a derived critical locus with weight 1 potential --- i.e. its quasi-smooth derived structure (and in particular its obstruction theory) arises from seeing it locally as
$$
\mathrm{Crit}(f)\ \subset\ U,
$$
where $U$ is a smooth ambient variety with a $\C^*$ action and $f\in\cO(U)$ is $\C^*$-equivariant with weight 1. The local model about a point $(\cE_\phi,s)$ is given by the following standard construction.

Suppose that $E=p_*\;\cE_\phi\in\cP^0_S$ has stabiliser group $G$. Then locally about $E$, the derived stack $\cP^0_S$ is isomorphic to the quotient by $G$ of the zero locus of a $G$-invariant section $\sigma$ of a $G$-equivariant vector bundle $F$ over a smooth ambient $G$-space $A$:
$$
\spreaddiagramcolumns{-2pc}
\xymatrix{
& F\ar[d] \\
Z(\sigma)\ \subset & A,\ar@/^{-1.5ex}/[u]_{\sigma} &&&&&& \cP^0_S\,\stackrel{\mathrm{loc}}=\,Z(\sigma)/G.}
$$
Near $(\cE_\phi,s)\in P_X\subset T^*[-1]\cP^0_S$, therefore, $T^*[-1]\cP^0_S$ inherits the following local description. We work in the bigger ambient space $\widetilde A:=F^*$, the total space of the dual of the vector bundle $F\to A$ over the old ambient space. On this we have the $G$-invariant function
$$
\widetilde\sigma\colon\widetilde A=F^*\To\C
$$
given by thinking of $\sigma\in\Gamma(F)$ as a linear functional on the fibres of $F^*$. By $G$-invariance its derivative lies in
\beq{dg}
d\;\widetilde\sigma\ \in\ \Gamma\Big(\!\ker\big(\Omega_{\widetilde A}\To\mathfrak g^*\big)\Big)
\eeq
and cuts out Crit$\,(\widetilde\sigma)\,\subset\,\widetilde A$. Dividing by $G$ gives the local model of $T^*[-1]\cP^0_S$, with its derived structure coming from thinking of \eqref{dg} as lying in the dg vector bundle $\Omega_{\widetilde A}\to\mathfrak g^*$.

Since the \emph{stable} pair $(\cE_\phi,s)$ has no automorphisms, the $G$-action is free there. In a neighbourhood then, $\Omega_{\widetilde A}\to\mathfrak g^*$ is onto with kernel $\Omega_{\widetilde A/G\,}$. Therefore over the open sub\emph{scheme} $P_X\subset T^*[-1]\cP^0_S$ we find the local description
\beq{model}
\spreaddiagramcolumns{-2pc}
\xymatrix{
& \Omega_{\widetilde A/G}\ar[d] \\
Z\big(d\;\widetilde\sigma\big)\ \subset\! & \widetilde A,\ar@/^{-1.5ex}/[u]_{d\;\widetilde\sigma} &&&&&& P_X\,\subset\,T^*[-1]\cP^0_S\,\stackrel{\mathrm{loc}}=\,Z\big(d\;\widetilde\sigma\big)/G.}
\eeq
This description makes $P_X$ into a $d$-critical scheme $(P_X,[\widetilde\sigma])$ in the sense of \cite{Jo}. Here $[\widetilde\sigma]$ is a section of Joyce's sheaf\;\footnote{For any local embedding of $P_X$ in a smooth ambient space $U$ with ideal $I$, it is the kernel of
$d\colon\frac{\sqrt I}{I^2}\ \To\ \frac{\Omega_U}{I\cdot\Omega_U}\,.$} $\cS^0_{P_X}$ on $P_X$ intrinsically defined by the scheme structure on $P_X$. In the above local chart $[\widetilde\sigma]$ is defined by simply restricting $\widetilde\sigma$ to the doubling of $P_X$ inside $\widetilde A/G$.

The $d$-critical scheme $(P_X,[\widetilde\sigma])$ is naturally ``\emph{oriented}": the determinant of the virtual cotangent bundle\footnote{In the language of Behrend-Fantechi, $\LL^{\mathrm{der}}$ is the perfect obstruction theory $E^\bullet$, which in turn is $T_{\widetilde A/G}\rt{d\widetilde\sigma}\Omega_{\widetilde A/G}$ in the local patch constructed above.} of its quasi-smooth derived structure
\beq{or}
\det\LL^{\mathrm{der}}_{T^*[-1]\cP^0_S}\big|_{P_X}\ =\ \Big(\!\det\LL^{\mathrm{der}}_{\cP^0_S}\big|_{P_X}\Big)^2
\eeq
has a natural square root given by the (pullback of) the determinant of the virtual cotangent bundle of the derived structure on $\cP_S^0$.

Since the potential function $\widetilde\sigma$ in the local model is $\C^*$-equivariant with weight 1 under the obvious scaling $\C^*$ action on $T^*[-1]\cP^0_S$ (equivalently the $\C^*$-action induced on $P_X$ by the usual one on $X=K_S$), the section $[\widetilde\sigma]\in\Gamma\big(\cS^0_{P_X}\big)$ defining the $d$-critical scheme structure $(P_X,[\widetilde\sigma])$ also has $\C^*$-weight 1. The $\C^*$-action preserves the orientation \eqref{or} and is circle compact (for any $p\in P_X$, there is a limit in $P_X$ of $\lambda\cdot p$ as $\lambda\in\C^*$ tends to 0).

Therefore we may apply the following result from \cite{Ma}. (It is an instance of dimensional reduction, and can be proved along the lines of the proof of \cite[Theorem 5.9]{BenSven}.)  We let $\overline\cM^{\hat\mu}_{P_X}$ denote the ring of $\hat\mu$-equivariant motives on $P_X$ \cite[Section 5]{BBBJ}, where $\hat\mu$ is the projective limit of the finite groups of roots of unity in $\C^*$. For any $\C^*$ fixed point $p\in P_X$ let
$$
\iota_p\colon U_p\ \Into\ P_X
$$
be the inclusion of the ascending cell
$$
U_p\ :=\ \big\{q\in P_X\colon\lim_{\C^*\ni\lambda\to0}\lambda\cdot q=p\big\},
$$
and let $\pi_p\colon U_p\to\{p\}$ be the projection.

\begin{thm} \textbf{\emph{\cite{Ma}}}
Let $MF\in\overline\cM^{\hat\mu}_{P_X}$ be the global motivic vanishing cycle of \cite[Corollary 5.17]{BBBJ}, defined by the oriented $d$-critical scheme structure $(P_n(X,\iota_*\beta),[\widetilde\sigma])$. Then
\beq{motres}
\pi_{p*\;}\iota_p^*\,MF\ =\ \mathbb L^{n_+/2},
\eeq
where $n_+$ is the rank of the positive weight part of the virtual tangent bundle $(\LL^{\mathrm{der}}_{P_X})^\vee\big|_p$ at $p$.
\end{thm}

\begin{cor}
The fixed points $p$ of $P_n(X,\iota_*\beta)$ have trivial Behrend function: $\chi^B\big|_p=(-1)^n$.
\end{cor}

\begin{proof}
To take Euler characteristics in \eqref{motres} we send $\LL^{1/2}$ to $-1$, yielding
\beq{amic}
\pi_{p*\;}\iota_p^*\,\chi^B\=(-1)^{n_+}.
\eeq
Since $U_p\take\{p\}$ has a $\C^*$ action without fixed points, it contributes 0 to $\pi_{p*\;}\iota_p^*\,\chi^B$. Therefore \eqref{amic} is just $\chi^B|_p$.

So it is left to show that $n_+\equiv n\pmod2$. In equivariant K-theory, the class of the virtual tangent bundle restricted to the point $(F,s)$ with $\chi(F)=n$ is
$$
R\Hom(I\udot,I\udot)\_0[1]\=R\Hom(\cO_X,F)+R\Hom(F,\cO_X)-R\Hom(F,F),
$$
where $I\udot$ is the complex $\cO_X\rt{s}F$. Decomposing $F=\oplus_{i=0}^d\,\iota_*F_i\,\t^{-i}$ into weight spaces, where $F_i$ is supported on $S$, we get
$$
\bigoplus_{i=0}^dR\Gamma(F_i)\;\t^{-i}-
\bigoplus_{i=0}^dR\Gamma(F_i)^*\t^{i+1}
-\bigoplus_{i,j=0}^dR\Hom(\iota_*F_i,\iota_*F_j)\;\t^{i-j}.
$$
Using adjunction and $\iota^*\iota_*F_i=F_i\,\oplus\;F_i\!\;\otimes\!K_S^{-1}[1]$ in K-theory, another application of Serre duality gives
\begin{multline*}
\bigoplus_{i=0}^dR\Gamma(F_i)\;\t^{-i}-
\bigoplus_{i=0}^dR\Gamma(F_i)^*\t^{i+1}
-\bigoplus_{i,j=0}^dR\Hom\_S(F_i,F_j)\;\t^{i-j} \\
+\bigoplus_{i,j=0}^dR\Hom\_S(F_i,F_j\otimes K_S)\;\t^{i-j+1}.
\end{multline*}
Since $K_S=\cO_S$, taking ranks of the positive weight pieces gives
\beqa
n_+ &=& -\sum_{i=0}^d\chi(F_i)-\sum_{i>j}\chi\_S(F_i,F_j)+\sum_{i\ge j}\chi\_S(F_j,F_i) \\
&=& -\chi(F)+\sum_{i=0}^d\chi\_S(F_i,F_i)\ \equiv\ n\pmod 2
\eeqa
because the intersection form of a K3 surface is even.
\end{proof}

\begin{cor}\label{kaichi}
The generating series of ``naive" Euler characteristic invariants
$$
\mathsf z^{\mathrm{na}}_P(X,q,v)\ :=\ \sum e\big(P_n(X,\iota_*\beta)\big)q^nv^\beta
$$
is more-or-less the same as $\mathsf z^\chi_P:$
$$
\mathsf z^{\mathrm{na}}_P(X,q,v)\=\mathsf z^\chi_P(X,-q,v).
\vspace{-6.5mm}$$
$\hfill\square$
\end{cor}

\section{Toda}
In this Section we will
combine our results so far with the KKV formula \cite{KKV, MP} proven in \cite{PT3}. Let $N_{g,\beta}^{\mathrm{red}}(X)$ denote the reduced connected residue Gromov-Witten invariants of $X$ defined by virtual $\C^*$-localisation. Rewriting their generating series
$$
\mathsf F^{\mathrm{red}}_{\mathrm{GW}}(X,u,v)\ :=\ \sum_{g\ge0,\,\beta\ne0}N_{g,\beta}^{\mathrm{red}}(X)u^{2g-2}v^\beta.
$$
in ``BPS form",
\beq{BPSform}
F_{\mathrm{GW}}^{\mathrm{red}}(X,u,v)\=
\sum_{g\geq 0,\,\beta\ne0}
 n_{g,\beta}\,u^{2g-2} \sum_{d>0}
\frac{1}{d}\left( \frac{\sin
({du/2})}{u/2}\right)^{2g-2}v^{d\beta},
\eeq
defines the Gopakumar-Vafa invariants $n_{g,\beta}\in\Q.$ Then by \cite{PT3} the $n_{g,\beta}$ are in fact integers $n_{g,h}\in\Z$ which depend only on $\beta$ through its self intersection
$$
\int_S\beta^2\=2h-2.
$$
The $n_{g,h}$ are nonzero only for $0\le g\le h$ and are determined by the KKV formula
\beq{kkvee}
\sum_{g\geq 0} \sum_{h\geq 0} (-1)^g n_{g,h}(y^{\frac{1}{2}} - y^{-\frac{1}{2}})^{2g}q^h\,=\,\prod_{n\geq 1} \frac{1}{(1-q^n)^{20} (1-yq^n)^2 (1-y^{-1}q^n)^2}
\eeq

\begin{thm}
The various enumerative invariants of $X$ are related by
\begin{eqnarray}\nonumber
\mathsf F_{\mathrm{GW}}^{\mathrm{red}}(X,u,v)\=\mathsf Z_P^{\red}(X,q,v) &=& -\log\big(1+\mathsf z_P^\chi(X,q,v)\big) \\
&=&-\log\big(1+\mathsf z^{\mathrm{na}}_P(X,-q,v)\big)\label{list}
\end{eqnarray}
under the substitution $q=-e^{iu}$. Furthermore, all are given by the KKV formula \eqref{kkvee}.
\end{thm}

\begin{proof}
In \cite[Corollary 4]{PT3} it is proven that
\beq{kk}
\mathsf Z_P^{\red}(X,q,v)\=\log\big(1+\mathsf Z_P^{\red}(T,q,v)\big),
\eeq
where $T$ is (an algebraic approximation to) the twistor space of the K3 surface $S$, and we sum over fibre classes $\alpha=(\beta,n)$ only.
We will combine this with the local Gromov-Witten/stable pairs correspondence conjecture proved in \cite[Theorem 2]{PT3},
\beq{wolf}
1+\mathsf Z_P^{\red}(T,q,v)\=\exp\big(\mathsf F_{\mathrm{GW}}(T,u,v)\big), \qquad q=-e^{iu},
\eeq
where
$$
\mathsf F_{\mathrm{GW}}(T,u,v)\ :=\ \sum_{g\ge0,\,\beta\ne0}N_{g,\beta}(T)u^{2g-2}v^\beta
$$
is the generating series of connected Gromov-Witten invariants of $T$.

We use
the trivial identity\footnote{The reason such a simple identity holds in Gromov-Witten theory is that $\C^*$-fixed stable maps to $X=S\times\C$ all live in the scheme-theoretic central fibre $S\times\{0\}$. In particular $\cM_{g,\beta}(T)=\cM_{g,\beta}(S\times\C)^{\C^*}$. This is not the case for stable pairs.}
$$
\mathsf F_{\mathrm{GW}}(T,u,v)\=\mathsf F_{\mathrm{GW}}^{\mathrm{red}}(X,u,v),
$$
proved by computing the fixed and moving parts of the $\C^*$-equivariant obstruction theory over $\cM_{g,\beta}(X)^{\C^*}\cong\cM_{g,\beta}(S)\cong\cM_{g,\beta}(T)$. 
Combined with \eqref{wolf} and \eqref{kk}, this gives
$\mathsf F^{\mathrm{red}}_{\mathrm{GW}}(X,u,v)\=\mathsf Z_P^{\red}(X,q,v)$. Theorem \ref{DM} then equates this with
$-\log(1+\mathsf z_P^\chi(X))$, and Corollary \ref{kaichi} implies the final claimed identity.
\end{proof}

The equality of the first and last terms of \eqref{list} confirms a  conjecture of Yukinobu Toda \cite{To1}, and proves that stable pair invariants of $X$ defined by weighted or unweighted Euler characteristic satisfy a form of the KKV conjecture of \cite{KKV,MP}. He actually asked whether we might have a chain of identities (up to signs and leading $1$s) like
$$
\exp\big(\mathsf F_{\mathrm{GW}}^{\mathrm{red}}(X)\big)\,\stackrel?=\,\mathsf Z_P^{\red}(X)\,\stackrel?=\,\mathsf z_P^\chi(X)
\,\stackrel?=\,\mathsf z^{\mathrm{na}}_P(X)
$$
instead of \eqref{list}. This is rather natural, since the first equality looks like the Gromov-Witten/stable pairs conjecture. But the reduced class means this holds only \emph{without} the exponential, which rather fortunately gets cancelled out again by Theorem \ref{it} at the second equality.
\medskip

For completeness we list all of the generating series which are proved to be equal and thus described by the KKV formula \eqref{kkvee}:
\begin{align*}
\mathsf F_{\mathrm{GW}}^{\mathrm{red}}(X)\;&\=\mathsf Z_P^{\red}(X)\=-\log\big(1+\mathsf z_P^\chi(X)\big)\=
-\log\big(1+\mathsf z_P^{\mathrm{na}'}(X)\big)
\\
&\=\mathsf F_{\mathrm{GW}}(T)\=\log\big(1+\mathsf Z_P^{\red}(T)\big)\=\log\big(1+\mathsf z_P^\chi(T)\big) \\
&\=\mathsf Z_P^{\red}(Y/E)\=\mathsf z_P^\chi(Y/E).
\end{align*}
Here $\mathsf z_P^{\mathrm{na}'}(X)$ denotes $\mathsf z_P^{\mathrm{na}}(X)$ with $q$ replaced by $-q$.\medskip

\noindent\textbf{BPS rationality.}
We have shown the weighted Euler characteristic stable pair invariants of $X$ satisfy
\beq{minus}
\log\big(1+\mathsf z_P^\chi(X,q,v)\big)\=-\mathsf F_{\mathrm{GW}}^{\mathrm{red}}(X,u,v).
\eeq
By the KKV formula \eqref{kkvee}, this can be written in BPS form \eqref{BPSform} with (for fixed  $\beta$) the $n_{g,\beta}$ nonzero only for finitely many $g$. As in \cite[Section 3.4]{PT1}, substituting $q=-e^{iu}$, this means that $\mathsf z_P^\chi$ can be written in the BPS form 
$$
-\sum_{g\ge0,\,\beta\ne0}\ \sum_{d\ge1}\,n_{g,\beta\,}\frac{(-1)^{g-1}}r\big((-q)^d-2+(-q)^{-d}\big)^{g-1}v^{d\beta}.
$$
Here the $n_{g,\beta}\in\Z,\ g\ge0$ (the Gopakumar-Vafa invariants of $T$) are the integers determined by the KKV formula \eqref{kkvee}.

Thus the generating series $\mathsf z_P^\chi(X,q,v)$ satisfies the ``BPS rationality" condition of \cite[Section 3.4]{PT1}. Toda \cite[Theorem 6.4]{To2} shows that this BPS rationality for $\mathsf z_P^\chi$ is \emph{equivalent} to a multiple cover formula for the Joyce-Song generalised DT invariants $J_X(r,\beta,n)$ \cite{JS}. These count Gieseker semistable torsion sheaves on $X=S\times\C$ in class $\iota_*(r,\beta,n)$ by weighted Euler characteristic and Joyce's Hall algebra logarithm (and turn out to be independent of the polarization used to define Gieseker semistability). We can therefore conclude that the multiple cover formula holds.

\begin{cor} The multiple cover formula
\beq{mc}
J_X(0,\beta,n)\=\sum_{k|(\beta,n)}\frac1{k^2}J_X(0,\beta/k,1)
\eeq
conjectured in \cite[Conjecture 6.20]{JS}, \cite{To2} holds for $X.\hfill\square$
\end{cor}

Furthermore in \cite[Section 6]{To2}, Toda uses \eqref{mc} to calculate the Joyce-Song invariants $J_X(v)$ for \emph{any} Mukai vector $v\in H^*(S,\Z)$. By combining deformations of $(S,v)$ with Fourier-Mukai equivalences of $D^b(S)$ he can assume $v$ is a curve class $(0,\beta,n)$ to which he can apply wall crossings (to show the invariants $J_X(v)$ are independent of stability condition) and then \eqref{mc} to handle any divisibility of $v$. Since \eqref{mc} is now proved, the result is the following.

\begin{cor} {\bf (Toda \cite{To2})} Let $v\in\big(\!\bigoplus_{p=0}^2H^{p,p}(S)\big)\cap H^*(S,\Z)$ be an algebraic class. For any polarization,
$$
J_X(v)\=\sum_{k|v}\frac1{k^2}\,e\!\left(\Hilb^{\frac12(v/k,v/k)+1}S\right),
$$
where $(\ \cdot\ ,\ \cdot\ )$ denotes the Mukai pairing.
$\hfill\square$
\end{cor}

\section{Vafa-Witten}
One can repeat the arguments of Sections \ref{Kai} and \ref{Georg} for Joyce-Song pairs instead of stable pairs. This has consequences for the Vafa-Witten invariants of \cite{TT2}, as we sketch now.

We begin on an arbitrary polarized surface $(S,\cO_S(1))$ 
and set $X$ to be the total space of the canonical bundle $K_S$ with its projection $p\colon X\to S$. (Later $S$ will revert to being a K3 surface.) For simplicity we assume as in \cite[Equation 2.4]{TT2} that the polarization on $S$ is \emph{generic} so that semistable sheaves are only destabilised by sheaves whose charges are proportional.

We consider compactly supported Gieseker semistable torsion sheaves $\cE$ on $X$ such that the charge of $p_*\;\cE$ is
$$
\alpha\=(r,c_1,c_2)\ \in\ \Z\oplus H^{1,1}(S,\Z)\oplus\Z.
$$
The sheaves $\cE$ can be described equivalently in terms of Higgs pairs $(E,\phi)$ on $S$ with $E:=p_*\;\cE$ and $\phi\in\Hom(E,E\otimes K_S)$; see \cite[Section 2]{TT1} for instance.

We then take Joyce-Song stable pairs $\cO_X\to\cE(n)$ for $n\gg0$, where $\cO_X(n):=p^*\cO_S(n)$. They form a fine moduli space $\cP_\alpha(X)$; see \cite[Section 3.1]{TT2} for a review. We let $$\cP^\perp_\alpha(X)\ \subset\ \cP_\alpha(X)$$ denote the moduli space of pairs whose underlying torsion sheaf $\cE$ has centre of mass zero on each $K_S$ fibre (equivalently $\tr\phi=0$ in the Higgs description) and $\det p_*\;\cE\cong\cO_S$.

Both moduli spaces carry $\C^*$-actions induced from the standard scaling action on the fibres of $K_S$. The two fixed loci are the same.

There is a reduced obstruction theory on $\cP_\alpha(X)$ which 
contains a summand $H^{\ge1}(K_S)$ governing the deformation-obstruction theory of the centre of mass of the sheaves (or trace of the Higgs field). This may be removed to define a \emph{symmetric} obstruction theory on $\cP^\perp_\alpha(X)$ \cite{TT2}. Therefore, on restriction to their common fixed locus, the former obstruction theory is the direct sum of the latter with $H^{\ge1}(K_S)\otimes\t$. Hence the two virtual normal bundles also differ by $H^{\ge1}(K_S)\otimes\t$.
Localising to their common $\C^*$ fixed locus then, we get the relation
\beq{resid}
P^\perp_\alpha(n)\ :=\ 
\int_{\big[(\cP^\perp_\alpha(X)^{\C^*}\big]^{\vir\ }}\frac1{e(N^{\vir})}\=\int_{\big[(\cP_\alpha(X)^{\C^*}\big]^{\mathrm{red}}}\,\frac{t^{h^0(K_S)-h^1(K_S)}}{e(N^{\mathrm{vir}})}
\eeq
for the reduced residue invariants $P^\perp_\alpha(n)$ of \cite[Section 7]{TT2}.

One can also consider invariants defined by weighted Euler characteristic. When $H^1(\cO_S)=0$  \cite[Section 4]{TT2},
\begin{multline} \label{wech}
P_\alpha(n)\,:=\,
e\big(\cP^\perp_\alpha(X),\chi^B\big)\,=\,
(-1)^{h^0(K_S)}e\big(\cP_\alpha(X),\chi^B\big) \\
=:\, (-1)^{h^0(K_S)}\widetilde P_\alpha(n).
\end{multline}
Again the sign is due to the extra deformations of $\tr\phi$.

In \cite{TT2} both sets of these pair invariants are studied in connection with \emph{Vafa-Witten invariants} of polarized surfaces $(S,\cO_S(1))$. These count semistable sheaves $\cE$ on $X=K_S$ with centre of mass zero on each $K_S$ fibre of $X\to S$ and $\det p_*\;\cE\cong\cO_S$.

The situation is most straightforward for the invariant $\vw_\alpha(S)$ defined in \cite[Section 4]{TT2} by Behrend localisation and 
Joyce's Hall algebra machinery. It is shown in \cite[Equation 4.2]{TT2} that the weighted Euler characteristics $P_\alpha$ of \eqref{wech} are related to the $\vw_\alpha$ by the following formula when $H^{0,1}(S)=0$.\footnote{When $H^{0,1}(S)\ne0$ there is a simpler formula \cite[Proposition 4.3]{TT2}, but it is not relevant for K3 surfaces.}
(Alternatively, we can take it to define the $\vw_\alpha$.) 
\beq{munch}
P_\alpha(n)\ =\ \mathop{\sum_{\ell\ge 1,\,(\alpha_i=\delta_i\alpha)_{i=1}^\ell:}}_{\sum_{i=1}^\ell\delta_i=1}
\frac{(-1)^\ell}{\ell!}\prod_{i=1}^\ell(-1)^{\chi(\alpha_i(n))} \chi(\alpha_i(n))\;\vw_{\alpha_i}(S).
\eeq

There is a rival invariant $\VW_\alpha(S)$ \cite{TT1} defined by virtual localisation instead of Behrend localisation when semistable sheaves of class $\alpha$ are all stable. The definition can be extended to the semistable case by its conjectural relationship to the virtual localisation pair invariants $P_\alpha^\perp$ of \eqref{resid}.

\begin{conj} \cite[Conjecture 7.2]{TT2} \label{conj}
If $H^{0,1}(S)=0=H^{0,2}(S)$ there exist $\VW_{\alpha_i}(S)\in\Q$ such that
$$
P^\perp_{\alpha}(n)\ =\ \mathop{\sum_{\ell\ge 1,\,(\alpha_i=\delta_i\alpha)_{i=1}^\ell:}}_{\sum_{i=1}^\ell\delta_i=1}
\frac{(-1)^\ell}{\ell!}\prod_{i=1}^\ell(-1)^{\chi(\alpha_i(n))} \chi(\alpha_i(n))\;\VW_{\alpha_i}(S)
$$
for $n\gg0$.
When either of $H^{0,1}(S)$ or $H^{0,2}(S)$ is nonzero we take only the first term in the sum:
\beq{shorter}
P^\perp_{r,L,c_2}(n)\ =\ (-1)^{\chi(\alpha(n))-1}\chi(\alpha(n))\VW_{r,L,c_2}(S).
\eeq
Furthermore we expect $\VW_\alpha=\vw_\alpha$ whenever $\deg K_S\le0$.
\end{conj}
These $\VW_\alpha$ define the Vafa-Witten invariants of $S$ when the conjecture holds. For instance by \cite[Proposition 6.8]{TT2} it holds when all semistable sheaves of charge $\alpha$ are stable, and in this case the $\VW_\alpha$ above reduce to the more direct definition of \cite{TT1}.

The conjecture also holds when $\deg K_S<0$ \cite[Theorem 6.14]{TT2}, and we prove it holds for $S$ a K3 surface in Theorem \ref{conjholds} below. In both of these cases $\VW_\alpha=\vw_\alpha$, but on general type surfaces the two invariants differ. There the virtual localisation invariant $\VW_\alpha$ seems to give the answers predicted by physics, while the weighted Euler characteristic invariant $\vw_\alpha$ gives the ``wrong" theory.
\medskip

From now on we fix $S$ to be a K3 surface.
Since $H^{0,2}(S)=\C$ we use the simplified formula \eqref{shorter} for the residue invariants $P^\perp_\alpha$, while continuing to use the full formula \eqref{munch} for weighted Euler characteristic invariants $P_\alpha$. We will see that \eqref{munch} is basically the exponential of \eqref{shorter}.

Like stable pairs, Joyce-Song pairs have no automorphisms (not even $\C^*$, unlike moduli of sheaves). Just as in Section \ref{Kai} they can therefore be patched canonically from $X=S\times\C$ into $Y=S\times E$ when their charge $\alpha$ is pushed forward from a K3 fibre $S$. The only difference is that to get \emph{stable} Joyce-Song pairs we patch only pairs whose underlying sheaves have the \emph{same reduced Hilbert polynomial}.  By the genericity of $\cO_S(1)$ this means we only patch pairs whose charges $\alpha$ are proportional. Thus we get the following analogue of Proposition \ref{jpo}
\beq{oof}
\sum e\big(\cP_\alpha(Y)/E,\chi^B\big)\;q^\alpha
\=-\log\Big(1+\sum e\big(\cP_\alpha(X),\chi^B\big)q^\alpha\Big),
\eeq
where both sums are over all $\alpha\ne0$ which are multiples of a fixed primitive class $\alpha_0$.

Section \ref{Georg} also goes through as before. As in \cite{Ob} the reduced obstruction theory on $\cP_\alpha(Y)$ gives a reduced cycle over which we can integrate an insertion to recover the left hand side of \eqref{oof}. The degeneration formula \cite{LW} again applies as we let $E$ acquire a nodal singularity. The same $\C^*$-localization argument from the proof of Theorem \ref{DM} thus gives
\beq{oof2}
\sum e\big(\cP_\alpha(Y,\iota_*\alpha)/E,\chi^B\big)\;q^\alpha
\=\sum P^\perp_\alpha(n)q^\alpha.
\eeq
Again both sums are over all $\alpha\ne0$ which are multiples of a fixed primitive class $\alpha_0$, and we have used \eqref{resid} to equate $P^\perp_\alpha(n)$ with the reduced residue invariants of $X$.
Combining \eqref{oof} with \eqref{oof2} gives the following.


\begin{prop}\label{VW} The reduced localised invariants \eqref{resid} are related to the weighted Euler characteristic invariants \eqref{wech} by 
$$
\sum P^\perp_\alpha(n)q^\alpha\=-\log\left(1+\sum\widetilde P_\alpha(n)q^\alpha\right),
$$
where both sums are over all $\alpha\ne0$ which are multiples of a fixed primitive class $\alpha_0. \hfill\square$
\end{prop}

Using the sign in \eqref{wech}, and setting $\mathbb N=\Z_{>0}$ without zero, this gives
\beqa
1-\sum_{\alpha\in\mathbb N\cdot\alpha_0}P_\alpha(n)q^\alpha &=& 
\exp\left(-\sum_{\alpha\in\mathbb N\cdot\alpha_0}P^\perp_\alpha(n)q^\alpha\right) \\
&=& 1+\sum_{\ell\ge1}\frac1{\ell!}\ \sum_{\alpha_1,\cdots,\alpha_\ell\in\mathbb N\cdot\alpha_0}\ \ \prod_{i=1}^\ell\left(-P^\perp_{\alpha_i}(n)q^{\alpha_i}\right),
\eeqa
by expanding the exponential. We conclude that
\beq{exp}
-P_\alpha(n)\=\mathop{\sum_{\ell\ge 1,\,(\alpha_i=\delta_i\alpha)_{i=1}^\ell:}}_{\sum_{i=1}^\ell\delta_i=1}\ \frac{(-1)^\ell}{\ell!}\prod_{i=1}^\ell P^\perp_{\alpha_i}(n).
\eeq

\begin{thm} \label{conjholds} Conjecture \ref{conj} holds when $S$ is a K3 surface, with
$$
\VW_\alpha(S)\=\vw_\alpha(S).
$$
\end{thm}

\begin{proof}
Comparing \eqref{exp} to \eqref{munch} shows that
$$
P^\perp_\alpha(n)\=-(-1)^{\chi(\alpha(n))}\chi(\alpha(n))\;\vw_\alpha(S).
$$
Therefore \eqref{shorter} is satisfied with $\VW_\alpha(S)=\vw_\alpha(S)$.
\end{proof}

\bibliographystyle{halphanum}
\bibliography{references}

%
%

\end{document}